\documentclass[11pt]{amsart}
\usepackage{amsmath}
\usepackage{amssymb}
\usepackage{amscd}
\usepackage{url}
\usepackage[all]{xypic}
\usepackage{pstricks}
\newcommand{\bp}{\begin{pmatrix}}
\newcommand{\ep}{\end{pmatrix}}
\newcommand{\ol}[1]{\overline{#1}}
\newcommand{\LS}{M}

\newcommand{\p}{\partial}
\newcommand{\sm}{\setminus}

\newcommand{\sph}{Sp^{\HVS}}
\newcommand{\spm}{Sp^{\MHS}}
\newcommand{\spt}{Sp^{\text{MHS}/2}}
\newcommand{\CC}{{\mathbb C}}

\newcommand{\RR}{{\mathbb R}}

\DeclareMathOperator{\nullt}{null}
\DeclareMathOperator{\lk}{lk}

\DeclareMathOperator{\sign}{signature}

\DeclareMathOperator{\MHS}{MHS}
\DeclareMathOperator{\HVS}{HVS}

\DeclareMathOperator{\Id}{Id}
\DeclareMathOperator{\intr}{int}
\DeclareMathOperator{\coker}{coker}

\numberwithin{equation}{section}
\numberwithin{equation}{subsection}

\theoremstyle{plain}
\newtheorem{theorem}[equation]{Theorem}
\newtheorem{lemma}[equation]{Lemma}
\newtheorem{proposition}[equation]{Proposition}

\theoremstyle{definition}

\newtheorem{remark}[equation]{Remark}

\newtheorem{definition}[equation]{Definition}

\newtheorem*{notation}{Notation}

\newtheorem{ass}[equation]{Assumption}

\newtheorem*{acknowledgements}{Acknowledgements}

\numberwithin{equation}{subsection}

\newcommand{\mv}{\mathcal{V}}

\newcommand{\mw}{\mathcal{W}}

\def\Z{\mathbb Z}

\def\Q{\mathbb Q}
\def\s{\sigma}
\def\S{\Sigma}
\def\C{\mathbb C}
\def\oS{S'}
\def\oM{M'}
\def\oSig{\S'}

\def\boleaspic{\begin{pspicture}(-6,-2.5)(6,2.5)
\pscircle(-3,0){2.0}
\psbezier(-1,2)(-1.5,2.3)(-1.5,1)(-2.5,1)
\psbezier(-2.5,1)(-1.5,1)(-1.5,0.0)(-2.8,0.0)
\psbezier(-2.8,0.0)(-1.5,0.0)(-1.5,-1)(-2.5,-1)
\psbezier(-2.5,-1)(-1.5,-1)(-1.5,-2.3)(-1,-2)
\rput(-1,-1.8){\psscalebox{0.8}{$X$}}
\pscircle[fillcolor=black,fillstyle=solid](-2.5,1){0.05}
\pscircle[fillcolor=black,fillstyle=solid](-2.8,0){0.05}
\pscircle[fillcolor=black,fillstyle=solid](-2.5,-1){0.05}
\pscircle[fillcolor=black,fillstyle=solid](-4.5,0){0.05}
\pscircle[linestyle=solid,linecolor=gray,fillstyle=none](-2.5,1){0.35}
\pscircle[linestyle=solid,linecolor=gray,fillstyle=none](-2.8,0){0.35}
\pscircle[linestyle=solid,linecolor=gray,fillstyle=none](-2.5,-1){0.35}
\pscircle[linestyle=solid,linecolor=gray,fillstyle=none](-4.5,0){0.35}
\rput(-2.4,1.45){\psscalebox{0.8}{$B_1$}}
\rput(-2.7,0.45){\psscalebox{0.8}{$B_2$}}
\rput(-2.4,-0.55){\psscalebox{0.8}{$B_3$}}
\rput(-4.4,0.45){\psscalebox{0.8}{$B_0$}}
\psbezier[linecolor=gray](-4.2,0.2)(-3.5,0.5)(-3.3,1)(-2.85,1)
\psbezier[linecolor=gray](-4.2,-0.2)(-3.5,-0.5)(-3.3,-1)(-2.85,-1)
\psbezier[linecolor=gray](-4.15,0)(-3.5,0.3)(-3.5,-0.3)(-3.15,0)
\rput(-3.5,0.9){\psscalebox{0.8}{$\gamma_1$}}
\rput(-3.5,0.15){\psscalebox{0.8}{$\gamma_2$}}
\rput(-3.5,-0.45){\psscalebox{0.8}{$\gamma_3$}}
\pscircle(4,0){2.0}
\psbezier(6,2)(5.5,2.3)(5.5,1)(4.5,1)
\psbezier(4.5,1)(5.5,1)(4.5,0.0)(4.2,0.0)
\psbezier(4.2,0.0)(5.5,0.0)(5.5,-1)(4.5,-1)
\psbezier(4.5,-1)(5.5,-1)(5.5,-2.3)(6,-2)
\rput(6,-1.8){\psscalebox{0.8}{$X$}}
\pscircle[fillcolor=black,fillstyle=solid](4.5,1){0.05}
\pscircle[fillcolor=black,fillstyle=solid](4.2,0){0.05}
\pscircle[fillcolor=black,fillstyle=solid](4.5,-1){0.05}
\pscircle[fillcolor=black,fillstyle=solid](2.5,0){0.05}
\psarc[linestyle=solid,linecolor=gray,fillstyle=none](4.5,1){0.35}{192}{168}
\psarc[linestyle=solid,linecolor=gray,fillstyle=none](4.2,0){0.35}{192}{168}
\psarc[linestyle=solid,linecolor=gray,fillstyle=none](4.5,-1){0.35}{192}{168}
\psarc[linestyle=solid,linecolor=gray,fillstyle=none](2.5,0){0.35}{55}{305}
\psbezier[linecolor=gray](2.7,0.3)(3.5,0.6)(3.7,1.1)(4.15,1.1)
\psbezier[linecolor=gray](2.8,0.1)(3.5,0.4)(3.7,0.9)(4.15,0.9)
\psbezier[linecolor=gray](2.8,-0.1)(3.5,-0.4)(3.7,-0.9)(4.15,-0.9)
\psbezier[linecolor=gray](2.7,-0.3)(3.5,-0.6)(3.7,-1.1)(4.15,-1.1)
\psbezier[linecolor=gray](2.75,0.1)(3.5,0.4)(3.5,-0.2)(3.85,0.1)
\psbezier[linecolor=gray](2.75,-0.1)(3.5,0.2)(3.5,-0.4)(3.85,-0.1)
\rput(3.5,1.0){\psscalebox{0.8}{$\oS$}}
\rput(2,1.0){\psscalebox{0.8}{$S$}}
\end{pspicture}}

\title[Semicontinuity of the mod 2 spectrum]{On the  semicontinuity of the mod 2  spectrum of hypersurface singularities}
\author{Maciej Borodzik}
\address{Institute of Mathematics, University of Warsaw, ul. Banacha 2,
02-097 Warsaw, Poland}
\email{mcboro@mimuw.edu.pl}
\author{Andr\'as N\'emethi}
\address{A. R\'enyi Institute of Mathematics, 1053 Budapest,  Re\'altanoda u. 13-15,  Hungary.}
\email{nemethi@renyi.hu}
\author{Andrew Ranicki}
\address{School of Mathematics, University of Edinburgh, Edinburgh EH9 3JZ, Scotland UK.}
\email{a.ranicki@ed.ac.uk}

\date{\today}
\makeatletter
\@namedef{subjclassname@2010}{\textup{2010} Mathematics Subject Classification}
\makeatother
\subjclass[2010]{primary: 57Q25 secondary: 57Q60, 14D07, 32S30}
\keywords{Codimension 2 embedding, Seifert surface,
Seifert matrix, Tristram-Levine signature, semicontinuity of the spectrum, variation structures,
Mixed Hodge Structure, link at infinity, higher dimensional knots}

\begin{document}
\begin{abstract}
We use purely topological methods to prove the semicontinuity of the mod 2 spectrum of local isolated hypersurface singularities in $\mathbb{C}^{n+1}$,
using  Seifert forms of high-dimensional non-spherical links,  the Levine--Tristram signatures and the generalized Murasugi--Kawauchi inequality.
\end{abstract}
\maketitle

\section{Introduction}

The present article extends the results of \cite{BN2}, valid for plane curves,  to arbitrary dimensions.
The main message is that the semicontinuity of the mod 2 Hodge spectrum (associated with local
isolated singularities, or with affine polynomials with some `tameness' condition) is topological in nature,
although its very definition and all known `traditional' proofs sit deeply in the analytic/algebraic theory.
Usually, the spectrum cannot be deduced from the topological data. In order to make clear these
 differences, let us review in short the involved invariants
 of a local isolated   hypersurface  singularity $f:(\CC^{n+1},0)\to(\CC,0)$.

The `homological package of the embedded topological type' of an isolated  hypersurface singularity
contains the information about the vanishing cohomology (cohomology of the Milnor fiber), the algebraic monodromy
acting on it, and different polarizations of it
including the intersection form and the (linking) Seifert form.
In fact, the Seifert form determines all this homological data. See e.g. \cite{AGV}.

On the other hand, the vanishing cohomology carries a mixed Hodge structure polarized by
the intersection form and the Seifert form, and it is  compatible with the monodromy action too.
It has several definitions, but all of them are analytic  \cite{A,Stee,St,Var,Var2}.
The equivariant  Hodge numbers were  codified by Steenbrink  in the spectral pairs; if one deletes the
information about the weight filtration one gets the spectrum/spectral numbers $Sp(f)$.
They are (in some normalization) rational numbers in the interval $(0,n+1)$.
Arnold  conjectured \cite{A}, and  Varchenko \cite{Var,Var2} and Steenbrink \cite{St} proved that
the spectrum behaves semicontinuously under deformations. In this way it becomes
a very strong tool e.g. in the treatment of the adjacency problem of singularities.

More precisely, in the presence of a deformation $f_t$,
where $t$ is the deformation parameter $t\in(\CC,0)$, the semicontinuity  guarantees that
 $|Sp(f_0)\cap I|\geqslant |Sp(f_{t\not=0})\cap I|$ for certain  semicontinuity domains $I$.
 Arnold  conjectured that $I=(-\infty,\alpha]$ is a semicontinuity domain for any $\alpha\in \RR$,
  Steenbrink and Varchenko proved the statement
 for $I=(\alpha,\alpha+1]$, which implies Arnold's conjecture. Additionally,  for some
 cases, Varchenko  verified   the stronger version, namely semicontinuity for $I=(\alpha,\alpha+1)$ \cite{Var}.

The relation between the Hodge invariants and the Seifert form was established  in \cite{Nem2},
proving that the collection of  mod 2 spectral pairs are equivalent with the real Seifert form.
Therefore, the real Seifert form determines  the mod 2 spectrum, that is,
the collection of numbers $\alpha$ mod 2 in $(0,2]$, where $\alpha$ run over $Sp(f)$.
Clearly, for plane curve singularities, i.e. when $n=1$, by taking mod 2 reduction we lose no information.

Surprisingly, this correspondence can be continued: in \cite{BN2} it is proved that if $n=1$ then
the {\it semicontinuity property of the mod 2 spectrum too is  topological}:
it  can be proved independently of Hodge theoretical tools, it follows from classical link theory.
Its precise statement
is the following:  length one `intervals' intersected by the mod 2 spectrum, namely sets of type
$Sp\cap (\alpha,\alpha+1)$ and $(Sp\cap (0,\alpha))\cup (Sp\cap (\alpha+1,2])$, for $\alpha\in [0,1]$,
satisfy semicontinuity properties, whenever  this question is well--posed (and under certain mild
extra assumptions regarding the roots of the monodromy operator).
The tools needed in this topological proof were the following: properties of the
{\it Tristram--Levine signature}, its connection with the spectrum, and the
{\it Murasugi--Kawauchi inequality} valid for it.

It was very natural to ask for a possible generalization of this fact for arbitrary dimensions.
This was seriously obstructed by two facts: the {\it existence of the Seifert form
up to an $S$--equivalence} and the analogue of the {\it Murasugi--Kawauchi inequality}
valid for any 2--codimension embedded manifold of $S^{2n+1}$.

This article has the following plan. First,  we collect and prove  all the other steps which are needed in the verification of  the 
semicontinuity, except these two main obstructions. Then, we handle the proofs of these two obstructions  at two different 
difficulty levels. 

The first level proof (the `strong version') is 
completely topological and uses  only the isotopy types of the corresponding higher dimensional links. 
It  is rather long and the needed  tools
are different from those needed to verify  the semicontinuity.
Therefore, we decided to separate them in another note  \cite{BNR}.

The second proof is adapted to the present application: it 
 assumes  the existence of  concrete Seifert  surfaces, which is hard to
prove in the general topological case,  but is automatically satisfied
in our present situation, since the links are cut out by holomorphic/polynomial equations. 
In this way the proof become considerable simpler and more transparent.
 This  `weak version' of  Murasugi--Kawauchi type
 inequality is given here in Subsection \ref{ss:2.1}. Some particularities of the links, e.g. the existence of the Milnor fiber, 
 which are rather standard in the theory of local singularities of affine polynomial maps, are used also several times.
 Therefore, from the point of view of the present application, the article can be consider complete and self--contained. 
 
\vspace{2mm}

The article is organized as follows.
In Section 2 we review the theory of higher dimensional links and their Seifert forms. Then we discuss the 
generalization of the Murasugi--Kawauchi inequality: we state the general version of it  from \cite{BNR}, and we give here a simpler proof using some additional assumptions (which are satisfied in the present case).  In the last subsection we
show how they can be applied for complex affine hypersurfaces.

Section 3 reviews facts about hermitian variation structures,
the tool which connects Seifert forms with the spectral numbers via the Levine--Tristram signatures.
Then we discuss how the machinery works in the case of a global affine hypersurface, provided that
the corresponding polynomial map satisfies some regularity conditions which guarantees
similar properties at infinity which are valid for local isolated singularities (the `Milnor package').
In this section we review some needed facts about these `tameness' conditions as well.

Section 4 contains the three semicontinuity results: (a)  the local case of
deformation of isolated singularities, (b) semicontinuity of the mod 2 spectrum of the
mixed Hodge structure at infinity of `nice' polynomials, and (c) an inequality which compares the
mod 2 spectrum at infinity of an affine fiber with the local spectrum of its singular points. 
We wish to emphasize that this type of semicontinuity (which produces bounds for the local spectra in terms of the
spectrum at infinity) is unknown in Hodge theory. 

Certain proofs and parts of the note show  similarities with \cite{BN2},
therefore some arguments are shortened, although we have tried to provide a presentation emphasizing
 all the basic steps of the proof.

\begin{notation}
For a finite set $A$, we denote by $|A|$ the cardinality of $A$, for a subset $X$ of $\CC^{n+1}$, $\intr X$ denotes its interior and $\ol{X}$ its
closure.
\end{notation}

\begin{acknowledgements}
The first author is supported by Polish MNiSzW Grant No N N201 397937 and also the
Foundation for Polish Science FNP.
The second author is partially supported by OTKA Grant K67928.
The authors express their thanks to the Research Fund of
the Edinburgh Mathematical Society for supporting the visit of MB and AN to Edinburgh in March 2012.
\end{acknowledgements}

\section{High dimensional links and their signatures}
\subsection{A quick trip through high dimensional link theory}\label{ss:2.1}
Here we present, in a condensed form, results about {\it high dimensional links},
that is, codimension 2 embeddings $M\subset S^{2n+1}$, where $M$ is a closed oriented $(2n-1)$--dimensional manifold, for any $n \geqslant 1$.
The theory resembles the `classical' theory of embeddings of disjoint
 unions of copies of $S^1$'s into $S^3$,
the special case $n=1$, although the proofs of the corresponding statements
are more involved. Unless specified otherwise, the results are from \cite{BNR}.
The first one, however, is older. It dates back to Erle \cite{Er}, see
also \cite{BNR}.

\begin{proposition}\label{prop:exofSei}
Let $M^{2n-1}\subset S^{2n+1}$ be a link. Then there exists a compact, oriented,
connected $2n$--dimensional manifold $\S\subset S^{2n+1}$, such that $\partial\S=M$.
\end{proposition}

We shall call such $\S$ a {\it Seifert surface}  for $M$; in general, there is not a
canonical choice of $\Sigma$, although there is a canonical choice if $M$ is the link of a
singularity. Given a Seifert surface $\S$, let $FH_n$ be the torsion free quotient of $H_n(\S,\Z)$.
Let $\alpha_1,\dots,\alpha_m$
be the basis of $FH_n$ and let us represent it by cycles denoted also by $\alpha_1,\dots,\alpha_m$.
We can define a matrix $S=\{s_{ij}\}_{i,j=1}^m$ with $s_{ij}=\lk(\alpha_i,\alpha_j^+)$, where $\alpha_j^+$ is a cycle $\alpha_j$ pushed slightly off $\S$
in the positive direction. $S$ is called the Seifert matrix of $M$ relative to $\Sigma$.

\begin{proposition}[See \expandafter{\cite[Main Theorem 2]{BNR}}]\label{prop:Sequiv}
Any two Seifert matrices of a given link $M$, corresponding to possibly different Seifert surfaces, are $S$-equivalent.
\end{proposition}

We refer to \cite[Chapter 5.2]{Kaw-book} for the definition $S$-equivalence of Seifert matrices.
The above result is standard for knots $S^{2n-1} \subset S^{2n+1}$ (when $b=S+(-1)^nS^T$ is
invertible), but is less familiar for links $M^{2n-1} \subset S^{2n+1}$.

Let us fix a link $M$ and let $\S$ be one  of its  Seifert surfaces, with Seifert matrix $S$. Following
an algebraic result of Keef \cite{Keef} (see also \cite[Section 3]{BN}), over the field of rational numbers
changing $S$ by a matrix $S$-equivalent to it we can write
\begin{equation}\label{eq:keef}
S=S_{ndeg}\oplus S_0,
\end{equation}
where $\det S_{ndeg}\neq 0$ and all the entries of $S_0$ are zero. Let us define
\begin{equation}\label{eq:n0}
n_0:=\textrm{size of }S_0=\dim(\ker S\cap \ker S^T).
\end{equation}
Moreover, the right hand side of (\ref{eq:keef})
does not depend on a particular choice of the Seifert matrix in its $S$-equivalence class,
more precisely, the integer $n_0$ and the   rational congruence class of the
matrix $S_{ndeg}$ are well defined.

\begin{definition}\label{def:mls}
Let $M^{2n-1} \subset S^{2n+1}$ be a link with a Seifert surface $\Sigma^{2n} \subset S^{2n+1}$,
$S$ the Seifert matrix and $S^T$ the transposed matrix.
The \emph{Alexander polynomial} of $M$ is defined by
\begin{align*}
\Delta_M(t)&:=\det(S_{ndeg}\cdot t + (-1)^nS_{ndeg}^T).
\end{align*}
For any  $\xi\in S^1\setminus \{1\}$ define
the  {\it Levine--Tristram signature} and {\it  nullity } of $M$ at $\xi$
\begin{align*}
\s_M(\xi)&:=\text{signature}\, [ \, (1-\xi)S+(-1)^{n+1}(1-\ol{\xi})S^T\, ]\\
n_M(\xi)&:=\text{nullity}\, [ \, (1-\xi)S+(-1)^{n+1}(1-\ol{\xi})S^T\, ]~.
\end{align*}
\end{definition}
\begin{remark}
Strictly speaking, $\Delta_M(t)$ is the classical Alexander polynomial only if $n_0=0$, i.e. if $S=S_{ndeg}$.
Otherwise, if $n_0>0$, then $\Delta_M(t)$ is the $(n_0+1)$--st classical Alexander polynomial.
In this note we shall not insist on this distinction.
\end{remark}
The following result is completely analogous to its one-dimensional version.
\begin{lemma}\label{lem:sigmaalex}
If $\xi\in S^1\setminus\{1\}$ is not a root of the Alexander polynomial, then $n_M(\xi)=n_0$.
Otherwise $n_M(\xi)>n_0$. For $\xi,\eta\in S^1\setminus\{1\}$ if  there exists an arc in
$S^1\setminus\{1\}$ connecting them and not containing any root
of the Alexander polynomial, then $\s_M(\xi)=\s_M(\eta)$.
\end{lemma}
For the completeness of the argument we present the straightforward proof.
\begin{proof}
Let $S=S_{ndeg}\oplus S_0$. Then $\s_M(\xi)=\sign[(1-\xi)S_{ndeg}+(-1)^{n+1}(1-\ol{\xi})S_{ndeg}^T]$
and $n_M(\xi)=\nullt[(1-\xi)S_{ndeg}+(-1)^{n+1}(1-\ol{\xi})S_{ndeg}^T]+n_0$. Hence we can assume $S=S_{ndeg}$.
But then,
\[(1-\xi)S+(-1)^{n+1}(1-\ol{\xi})S^T=(\ol{\xi}-1)(\xi S+(-1)^nS^T).\]
Hence $\Delta_M(\xi)\neq 0$ if and only if $(1-\xi)S+(-1)^{n+1}(1-\ol{\xi})S^T$ is non-degenerate.
Furthermore, if the arc $\gamma$ joins $\xi$ and $\eta$,
 then the matrix $[(1-\alpha)S+(-1)^{n+1}(1-\ol{\alpha})S^T]_{\alpha\in\gamma}$ is a path
in the space of non-degenerate sesquilinear forms
along which the  signature is constant.
\end{proof}

Our main tool  is the following generalization of the classical {\it Murasugi--Kawauchi inequality}.

\begin{theorem}[see \expandafter{\cite[Main Theorem 3]{BNR}}]\label{thm:murineq}
Let $(Y;M_0,M_1) \subset S^{2n+1}\times([0,1];\{0\},\{1\})$ be a cobordism of links $M_0,M_1 \subset S^{2n+1}$.
For any Seifert surfaces $\S_0,\S_1 \subset S^{2n+1}$ for $M_0,M_1$ and $\xi\in S^1\setminus\{1\}$ we have
\begin{equation}\label{eq:murineq}
|\sigma_{M_0}(\xi)-\sigma_{M_1}(\xi)|\leqslant b_n(\Sigma_0\cup_{M_0}Y\cup_{M_1}\Sigma_1)-b_n(\S_0)-b_n(\S_1)
+n_{M_0}(\xi)+n_{M_1}(\xi),
\end{equation}
where $b_n(\cdot)$ denotes the $n$-th Betti number.
\end{theorem}

\begin{remark}
In the classical case $n=1$, the inequality looks slightly different (compare \cite[Theorem~12.3.1]{Kaw-book}), namely one has 
\begin{equation}\label{eq:murineq2}
|\sigma_{M_0}(\xi)-\sigma_{M_1}(\xi)|+
|n_{M_0}(\xi)-n_{M_1}(\xi)| \leqslant b_1(Y,M_0)=b_1(Y,M_1)~.
\end{equation}
In this case one has a better interplay between the topology and nullities (see for example the quantities $w_L$ and $u_L$ in \cite[Section 5]{Bo}).
We do not know whether this stronger inequality
holds in higher dimensions; the approach of \cite{BNR} seems to be insufficient to prove that. 
\end{remark}

For the convenience of the reader, and also for the completeness of the proof of the semicontinuity property,
we present  below a sketch of proof of Theorem~\ref{thm:murineq} under the following additional
assumption.

\begin{ass}\label{ass:ass}
There exists a smooth function $f\colon S^{2n+1}\times [0,1]\to\C$ such that:
\begin{itemize}
\item[$\bullet$] $Y=f^{-1}(0)$, in particular $M_0=S^{2n+1}\times\{0\}\cap f^{-1}(0)$ and $M_1=S^{2n+1}\times\{1\}\cap f^{-1}(0)$;
\item[$\bullet$] the map $\arg f\colon S^{2n+1}\times[0,1]\sm Y\to S^1$ given by $x\mapsto \frac{f(x)}{|f(x)|}$ is surjective;
\item[$\bullet$] there is a non-critical value $\delta\in S^1$ of $\arg f$ and of the restriction of $\arg f$ to $S^{2n+1}\times\{0,1\}\sm(M_0\sqcup M_1)$,
such that if we define $\Omega=(\arg f)^{-1}(\delta)$ then $\S_i=\Omega\cap S^{2n+1}\times\{i\}$ for $i=0,1$.
\end{itemize}
\end{ass}

If we study deformations of isolated singularities of hypersurfaces, the existence of function $f$ will always be clear.
The function will be then even complex analytic,
although we only need smoothness to ensure that the Sard's lemma works.
In fact, many arguments
of \cite{BNR} rely on constructing a function like the one above using homological methods.
The last condition of the assumptions is
slightly weaker than the condition that $\arg f$ gives a fibration of $S^{2n+1}\times\{i\}\sm M_i\to S^1$ with fiber isotopic
to $\S_i$ for $i=0,1$.

Assumption~\ref{ass:ass} is satisfied in all the cases that we consider in the paper (see Remark~\ref{rem:ass} or 
Proposition~\ref{prop:newSeifert} below).
That enables us to prove all the results from Section~\ref{s:4} without referring to surgery theory from \cite{BNR}. 

However, in this
approach, the signatures and their properties depend \emph{a priori} on the function $f$ and on the specific
choice of the Seifert surface.
It is not clear whether these properties are of topological nature, in particular, whether the results from Section~\ref{s:4}
are purely topological, or not.
The theorems from \cite{BNR} clarify this. The Levine--Tristram signatures
are defined even if the link in question is not fibered and depend only on the isotopy of the link.
Moreover, the behaviour of the Levine--Tristram signatures under cobordism depends only on homological properties
of the manifold which realizes the cobordism between links.

\begin{proof}[Proof of Theorem~\ref{thm:murineq}]
We denote $C_0:=S^{2n+1}\times\{0\}$ and $C_1:=S^{2n+1}\times\{1\}$.  On $(-C_0)\sqcup C_1$ (the minus
sign denotes that we reverse the orientation) we can
consider the linking form: if $\alpha,\beta$ are two {\it disjoint} $n$-dimensional cycles on $C_0\sqcup C_1$ we define
\[\lk(\alpha,\beta)=-\lk(\alpha_{0},\beta_0)+\lk(\alpha_1,\beta_1),\]
where $\alpha_i=\alpha\cap C_i$, $\beta_i=\beta\cap C_i$, $i=0,1$. This definition allows us to define the Seifert pairing for $M_0\sqcup M_1$
with respect to the Seifert surface $-\S_0\sqcup\S_1$ by $S(\alpha,\beta)=\lk(\alpha,\beta^+)$, where $\beta^+$
is the cycle $\beta$ pushed off slightly from $\S_0\sqcup\S_1$ in a positive normal direction. If $S_0$ and $S_1$ denote
the Seifert pairings for $M_0$ and $M_1$, then clearly
$S=(-S_0)\oplus S_1$.

Observe now that $\p\Omega=\S_0\cup Y\cup\S_1$. Let $k\colon\S_0\sqcup\S_1\to\p\Omega$ be the inclusion.
We denote also $j\colon\p\Omega\to\Omega$ and $i=k\circ j$.

\smallskip

{\bf Claim.} If $\alpha,\beta$ are disjoint and are elements of
$\ker i_*\colon H_n(\S_0\sqcup \S_1;\Q)\to H_n(\Omega;\Q)$, then $S(\alpha,\beta)=0$.

\smallskip
To prove the claim, we assume that $\alpha,\beta\in\ker i_*$. Then, there exist $(n+1)$-dimensional cycles $A,B\subset \Omega$, such
that $\p A=\alpha$, $\p B=\beta$. Let $B^+$ be the cycle in $S^{2n+1}\times[0,1]$
obtained by pushing $B$ off $\Omega$ in a positive normal direction.
Clearly $\p B^+=\beta^+$. But then
\[\lk(\alpha,\beta^+)=A\cdot B^+=0,\]
because $A$ and $B^+$ are disjoint. This proves the claim.

\smallskip
Let now $k_*$ and $j_*$ denote the induced maps on $n$-th homology with rational coefficients.
By a standard Poincar\'e duality argument $\dim\ker j_*=\frac12b_n(\p\Omega)$. Therefore
\[\dim\ker i_*\geqslant \dim\ker k_*+(\dim\ker j_*-\dim\coker k_*)=\frac12 b_n(\p\Omega)-(b_n(\p\Omega)-b_n(\S_0\sqcup \S_1)).\]
Consider now $\xi\in S^1\sm\{1\}$ and assume for simplicity that $n_{M_0}(\xi)=n_{M_1}(\xi)=0$ (the proof in
general case is only slightly more complicated). This means that the form $(1-\xi)S+(1-\ol{\xi})S^T$
is non-degenerate. By the claim, it vanishes on a subspace of dimension $\dim\ker i_*$, therefore the absolute value of its signature is bounded by
\[b_n(\S_0\sqcup \S_1)-2\dim\ker i_*=b_n(\p\Omega)-b_n(\S_0\sqcup \S_1).\]
\end{proof}

\subsection{Application to affine or local complex hypersurfaces}
Let $X\subset \mathbb{C}^{n+1}$ be a complex hypersurface with at most
isolated singularities.
If $x\in X$ is a singular point of $X$, consider a sufficiently small sphere
$S_x\simeq S^{2n+1}$ centered at $x$. The intersection $X\cap S_x$ embedded in $S_x$ is called
the \emph{link} of the hypersurface singularity $(X,x)\subset ( \mathbb{C}^{n+1},x)$.
We shall denote it by $\LS_x\subset S_x$.

\smallskip
Let now $B$ be any ball in $\mathbb{C}^{n+1}$ such that $S=\partial B$ is transverse to $X$ (in particular $S$ omits the
singular set of $X$).
Let $M:=S\cap X\subset S$. It is a closed codimension 2 submanifold of $S$, i.e. a high dimensional link.
Let $x_1,\dots,x_k$ be those singular points of $X$, that lie inside $B$.
We have the following result.
\begin{proposition}\label{prop:bol}
Let $X^s$ be the smoothing of $X$ inside $B$ (whose boundary will be identified via an isotopy with
$M$ too) and $ \Sigma$ a Seifert surface of $M\subset S$.
Moreover, for $j=1,\dots,k$, let $ \Sigma_j$ be the Milnor fiber at $x_j$.
 Then for all $\xi\in S^1\setminus\{1\}$
\[\Big|\sigma_{M}(\xi)-\sum_{j}\sigma_{\LS_{x_j}}(\xi)\Big|\leqslant
b_n(X^s\cup_M \Sigma)-b_n( \Sigma)-\sum_{j}b_n( \Sigma_j)+
n_M(\xi)+\sum_{j} n_{\LS_{x_j}}(\xi).\]
\end{proposition}
\begin{figure}
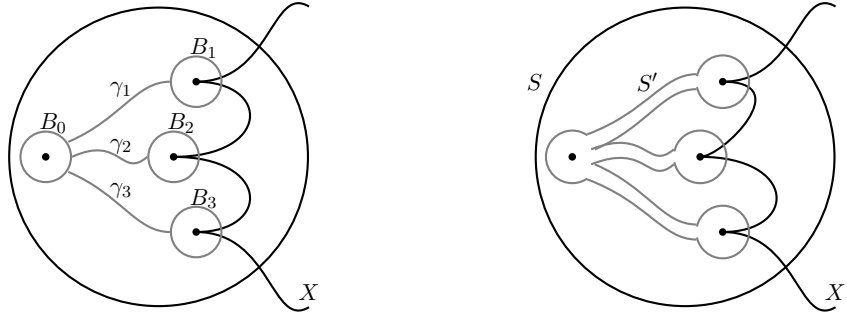

\boleaspic
\caption{The cobordism between local and global links}\label{fig:boleas}
\end{figure}
\begin{proof}
For any $j=1,\dots,k$ let us pick a small Milnor balls $B_j$ around $x_j$ such that the balls are
pairwise disjoint; hence $\partial B_j\pitchfork X$, $\partial B_j\cap X=\LS_{x_j}$ and
$ \Sigma_j\subset\partial B_{j}$.
Let $B_0\subset B$ be a ball disjoint from $X\cup B_1\cup\dots\cup B_k$.
For each $j=1,\dots,k$, let $\gamma_j$ be a smooth, closed curve joining $\partial B_0$ with $\partial B_k$,
such that $\gamma_j$ is disjoint from $ \Sigma_j$ and from all other balls $B_{l}$
 and other curves $\gamma_l$ for $l\neq j$. We also assume that the relative
 interior of $\gamma_j$ is disjoint from $B_0$ and $B_j$ (see Figure~\ref{fig:boleas}).
Let $U_j$ be a small tubular neighbourhood of $\gamma_j$ and consider
\[U=B_0\cup\bigcup_{j}(B_j\cup U_j), \ \ \ \mbox{and} \ \ \  Y=X\cap B\setminus
\bigcup_{j}\intr B_j.\]

The assumptions on $\gamma_j$'s guarantee that
\begin{itemize}
\item $U$ is diffeomorphic to a standard ball;
\item $\oS:=\partial U$ (after possibly smoothing corners) is a sphere transverse to $X$;
\item $\oM:=\oS\cap X$ is a disjoint union
  $\LS_{x_1}\sqcup\dots\sqcup\LS_{x_k}$ of the local links;
\item $ \oSig:= \Sigma_1\sqcup\dots\sqcup  \Sigma_k$ is a Seifert surface for $\oM$.
\end{itemize}
In particular
\[\sigma_{\oM}(\xi)=\sum_{j}\sigma_{\LS_{x_j}}(\xi)\ \text{ and } \
n_{\oM}(\xi)=\sum_{j} n_{\LS_{x_j}}(\xi).\]
We say that the cobordism of links
$$(Y;M,\oM) \subset
(B\setminus \intr U;S,\oS)~\approx~
S^{2n+1} \times ([0,1];\{0\},\{1\})$$
is constructed by the `boleadoras' trick. The generalized Murasugi--Kawauchi inequality of Theorem~\ref{thm:murineq} gives
\[
\Big|\sigma_{M}(\xi)-\sum_{j}\sigma_{\LS_{x_j}}(\xi)\Big|\leqslant b_n(Y\cup\oSig\cup \Sigma)
-b_n(\oSig)-b_n( \Sigma)+n_{M}(\xi)+\sum_{j}n_{\LS_{x_j}}(\xi).\]
Now $\Sigma_j\approx X^s\cap B_j$ by \cite{Milnor-book}, hence
$Y\cup\oSig\approx X^s$, and the statement follows.
\end{proof}
\begin{remark}\label{rem:ass} (a) \ 
If $X$ is given by $\{f=0\}$ for a complex polynomial $f$ and there exists $\delta\in S^1$ such that $\S=(\arg f)^{-1}(\delta)$ (for instance if
the link $M$ is fibred), then we can find $B_0$ and the paths $\gamma_1,\dots,\gamma_k$ (used in the proof of Proposition~\ref{prop:bol})
disjoint from $(\arg f)^{-1}(\delta)$. Then $(\arg f)^{-1}(\delta)\cap S'$ is a disjoint union of the Seifert surfaces for $M_{x_1},\dots,M_{x_k}$.
This proves that Assumption~\ref{ass:ass} in this case is satisfied.

(b) \  Proposition~\ref{prop:bol} is the main tool in the proof of the semicontinuity
of the mod 2 spectrum, cf. Section \ref{s:4}. We remark that if $n=1$, this proposition
allows us to prove semicontinuity of the spectrum
without referring to \cite{Bo} and \cite{BNR}. (The article \cite{BN2} uses \cite{Bo}.)

(c) \ 
The space $U$, as drawn in the picture above, resembles
the South American throwing weapon \emph{boleadoras} \cite{Bol}, hence the name of the cobordism construction
in \ref{prop:bol}.
\end{remark}

\section{Hermitian Variation Structures and Mixed Hodge Structures}
\subsection{Generalities about hermitian variation structures}\label{ss:hvslinks}
Variation structures were introduced in \cite{Nem2}. As it was shown in \cite{BN} and \cite{BN2} they form a bridge between knot theory and Hodge theory.
Let us recall shortly the definition, referring to \cite{Nem2} or \cite[Section 2]{BN} for all details and further references.

\begin{definition}
Fix a sign $\varepsilon=\pm 1$.
An \emph{$\varepsilon$--hermitian  variation structure} (in short: HVS) consist of the quadruple
$(U;b,h,V)$, where $U$ is a complex linear space,
$b\colon U\to U^*$ is an $\varepsilon$--hermitian endomorphism (it can be regarded as
a $\varepsilon$--symmetric pairing on $U\times U$),
$h\colon U\to U$ is an automorphism preserving $b$, and $V\colon U^*\to U$
is an endomorphism such that
\[V\circ b=h-I\text{ and }\ol{V}^*=-\varepsilon V\circ \ol{h}^*.\]
Here $\ol{\cdot}$ denotes the complex conjugate and $*$ the duality.
\end{definition}
We shall call a HVS \emph{simple} if $V$ is an isomorphism. In this case $V$ determines $b$ and $h$ completely by the formulae $h=-\varepsilon V(\ol{V}^*)^{-1}$
and $b=-V^{-1}-\varepsilon {\ol{V}^*}^{-1}$. It was proved in \cite{Nem2} that each simple variation structure is a direct sum of indecomposable ones,
moreover the decomposable
ones can be completely classified. More precisely, for any $k\geqslant 1$ and any $\lambda\in\mathbb{C}$  with $|\lambda|=1$,  there are two
 structures $\mw_{\lambda}^k(\pm 1)$ (up to an isomorphism). In their  case
$h$ is a single Jordan block of size $k$ with eigenvalue $\lambda$. Furthermore, for any $k\geqslant 1$ and any
$\lambda\in\mathbb{C}$ such that $0<|\lambda|<1$ there exists a single structure $\mv_{\lambda}^{2k}$. In this case
$h$ is a direct sum of two Jordan blocks of size $k$: one block has eigenvalue $\lambda$, the other $1/\bar{\lambda}$. In particular for each HVS $\mv$ there exists
a unique decomposition
\begin{equation}\label{eq:decomp}
\mv=\bigoplus_{\substack{0<|\lambda|<1\\ k\geqslant 1}} q^k_\lambda\cdot   \mv^{2k}_\lambda
\oplus
\bigoplus_{\substack{|\lambda|=1\\ k\geqslant 1, \ u=\pm 1}} p^k_\lambda(u)\cdot  \mw^{k}_\lambda(u),
\end{equation}
where $q^k_\lambda$ and $p^k_\lambda(\pm 1)$ are certain non--negative integers. Here we write $m\cdot\mv$ for a
direct sum of $m$ copies of $\mv$. Next
we recall the definition of the spectrum associated with  a HVS.
\begin{definition}
Let $\mv$ be a HVS. Let $p^k_\lambda(\pm 1)$ and $q^k_\lambda$ be the integers defined by \eqref{eq:decomp}.
The {\it extended spectrum} $ESp$ is the union $ESp=Sp\cup ISp$, where
\begin{itemize}
\item[(a)] $Sp$, the {\it spectrum},  is a finite set of real numbers from the interval $(0,2]$ with integral multiplicities such that any real
number $\alpha$ occurs in $Sp$ precisely $s(\alpha)$ times, where
\[s(\alpha)=\sum_{n=1}^\infty\sum_{u=\pm1}\left(\frac{2n-1-u(-1)^{\lfloor \alpha\rfloor}}
{2}p^{2n-1}_\lambda(u)+np^{2n}_\lambda(u)\right), \ \
(e^{2\pi i\alpha}=\lambda).\]
\item[(b)] $ISp$ is the set of complex numbers from $(0,2]\times i\mathbb{R}$,
$ISp\cap\mathbb{R}=\emptyset$, where $z=\alpha+i\beta$ occurs
in $ISp$ precisely $s(z)$ times, where
\[
s(z)=\begin{cases}
\sum k\cdot q^k_\lambda&\text{if $\alpha\leqslant1$, $\beta>0$ and $e^{2\pi i z}=\lambda$}\\
\sum k\cdot q^k_\lambda&\text{if $\alpha>1$, $\beta<0$ and $e^{2\pi i z}=1/\bar{\lambda}$}\\
0&\text{if $\alpha\leqslant 1$ and $\beta<0$, or $\alpha> 1$ and $\beta>0$}.
\end{cases}
\]
\end{itemize}
\end{definition}
We have the following relation
\[|ESp|=\dim U=\deg\det(h-t\Id).\]

Main motivation for introducing HVS comes from singularity theory, cf. \cite{AGV,Nem2}.

\begin{lemma}\label{lem:hvsofsing}
Let $X\subset\mathbb{C}^{n+1}$ be a complex hypersurface with an isolated singularity
$x\in X$. Let $S_x$ be a small sphere centered at $x$ and
let $\pi\colon S_x\setminus X\to S^1$ be the Milnor fibration with fiber $ \Sigma$.
Let $U=H_n( \Sigma;\mathbb{C})$, $b$ be the intersection form on $U$
and $h\colon U\to U$ be the homological monodromy.
Finally, let $V\colon H_n( \Sigma,\partial  \Sigma;\mathbb{C})\to H_n( \Sigma;\mathbb{C})$ be the
Picard--Lefschetz variation
operator. Then the quadruple $(U;b,h,V)$ constitutes a simple HVS with $\varepsilon=(-1)^n$.
\end{lemma}

\begin{definition}
The HVS defined in Lemma~\ref{lem:hvsofsing} is called the HVS of $(X,x)\subset (\mathbb{C}^{n+1},x)$.
\end{definition}

\begin{remark}
The Seifert matrix (associated with $\Sigma$) is the inverse transpose of $V$.

Conversely, any  non-degenerate Seifert matrix $S$ associated with any (topologically defined)
link $M\subset S^{2n+1}$ determines a simple HVS via $V=(S^{-1})^T$.
\end{remark}

In the above algebraic/analytic  case, by the Monodromy Theorem,
$h$ has all eigenvalues on the unit circle. In particular $ESp=Sp$.
The spectrum associated with
the {\it hermitian  variation structure} of an isolated singularity will be denoted by  $\sph$.

On the other hand, for an isolated singular point $x$ one can define a {\it mixed
Hodge structure} on the $n$-th cohomology of its Milnor fiber, which determines  a spectrum denoted by $\spm$,
cf. work of Steenbrink and Varchenko \cite{Stee,St,Var}. (We will use the same normalization as in
\cite[2.3]{BN2}. In particular, $\spm$ is a subset of $(0, n+1)$.)

The relation between the two sets of spectral numbers is given in Proposition~\ref{prop:mod2red}. Before
we state it, we need to fix some terminology.
\begin{definition}\label{def:specmod2}
We denote
the mod\,2 reduction of the spectrum $\spm$   by $\spt$.
This means that $\spt$ is a  finite set of  rational  numbers from the interval $(0,2]$ with integral
multiplicities, and
for any $\alpha\in (0,2]$ the multiplicity of the spectral number in
$\spt$ is the sum of the multiplicities of the spectral numbers $\{\alpha+2j\}_{j\in\mathbb{Z}}$ in $\spm$.
\end{definition}

\begin{proposition}[\expandafter{\cite[Theorem~6.5]{Nem2}}]\label{prop:mod2red}
 $\sph$ is the mod\,2 reduction of $\spm$ for any isolated hypersurface singularity.
\end{proposition}

Note that the spectrum $\sph$, by its very definition,  can  be  determined from the Seifert form
of the link. The point is that it can be fully recovered from the signatures as well. 
\begin{proposition}[\expandafter{\cite[Corollary~2.4.6]{BN2}}]\label{prop:sing}
Let $x$ be an isolated singular point of the hypersurface $X\subset\mathbb{C}^{n+1}$,
and $\mv=(U;b,h,V)$ the corresponding HVS.
Let $\alpha\in[0,1)$ be such that $\xi=e^{2\pi i\alpha}$ is not an eigenvalue of the monodromy operator $h$.
Then we have
\begin{align*}
|\sph\cap (\alpha,\alpha+1)|&=\frac12(\dim U-\sigma_{\LS_x}(\xi))\\
|\sph\setminus [\alpha,\alpha+1]|&=\frac12(\dim U+\sigma_{\LS_x}(\xi)).
\end{align*}
\end{proposition}
\begin{remark}\label{rem:nulnonzero}
The dimension $\dim U=b_n( \Sigma)$ is the Milnor number of the singularity.
Moreover, the condition that $\xi$ is not eigenvalue of the monodromy is equivalent to  $n_{\LS_x}(\xi)=0$.
This follows from  Lemma~\ref{lem:sigmaalex} and from the fact that the
 Seifert matrix of the link corresponding to the Milnor fiber is non-degenerate.
\end{remark}

In Proposition~\ref{prop:sing} we assume that $\alpha\in[0,1)$. If $\alpha=1$, the statement of proposition still holds. In fact
$\sph\cap(1,2)=\sph\setminus[0,1]$ and $\sph\setminus[1,2]=\sph\cap(0,1)$. Hence the case $\alpha=1$ is equivalent to the case
$\alpha=0$.

\subsection{Spectrally tame polynomials.}\label{ss:stp}
We introduce now a new class of tame polynomials, namely {\it spectrally tame}.
We add this new terminology (to the rather big variety of different versions of
`tame' polynomials) in order to make precise, what assumptions are needed to obtain
a topological proof for the semicontinuity of the mod\,2 reduction of the MHS-spectrum at infinity associated with
 polynomial maps.  Below, we
shall give some examples of spectrally tame polynomials as well.

Let $P\colon\mathbb{C}^{n+1}\to\mathbb{C}$ be a polynomial map.
Let ${\mathcal B}$ be the bifurcation set of $P$, a finite subset
${\mathcal B}\subset \mathbb{C}$ such that the restriction of $P$ is
a $C^\infty$ locally trivial fibration over $\mathbb{C}\setminus {\mathcal B}$.
Let $D\subset \mathbb{C}$ be a sufficiently large closed disc centered at the origin
so that ${\mathcal B}\subset D$. Finally, take $R$ sufficiently large, such that for any $R'\geqslant R$
the boundary of any closed ball $B_{R'}\subset \mathbb{C}^{n+1}$ centered at the origin intersects
any $P^{-1}(c)$, $c\in\partial D$,  transversally.

For a fixed  value $c\in \partial D$ set $X_c:=P^{-1}(c)$, the generic fiber of $P$,
 and let $M_c$ be the link at infinity $X_c\cap \partial B_R$. Notice that by the above choices,
 the restriction of $P$ on $P^{-1}(\partial D)$, or on $P^{-1}(\partial D)\cap \partial B_R$
 is the fibration of $P$ at infinity.

Since $X_c$ is Stein, $H_i(X_c,\mathbb{Z})=0$ for $i>n$ and $H_n(X_c,\mathbb{Z})$ is free.
Moreover, $H^n(X_c,{\mathbb Q})$ carries a MHS, the {\it `mixed Hodge structure of $P$ at infinity'},
see e.g. \cite{Di,NS,Sa}. We add to these facts the following.

\begin{proposition}\label{prop:newSeifert}
Consider $\phi(z):=P(z)/|P(z)|:\partial B_R\setminus P^{-1}(\intr D)\to S^1$,
the Milnor map restricted on the complement of  $P^{-1}(\intr D)$. Then
$$\phi:(\partial B_R\setminus P^{-1}(\intr D), \partial B_R\cap P^{-1}(\partial D))\to S^1$$
is a $C^\infty$ locally trivial fibration of pairs of spaces over $S^1$. This fibration is $C^\infty$
equivalent with the fibration at infinity associated with $P$
$$P:(P^{-1}(\partial D)\cap B_R,P^{-1}(\partial D)\cap \partial B_R)\to \partial D.$$
In particular, $M_c\subset \partial B_R$ admits a Seifert surface, namely $\Sigma:=\phi^{-1}(c/|c|)$,
which is diffeomorphic to $X_c\cap B_R$.
\end{proposition}
\begin{proof}
The proof is similar to the proofs of Theorems 10 and 11 from \cite{NZ2}, with the only modification that
 the arbitrary disc $D$ used in [loc.cit.] for semitame polynomials should be replaced by a sufficiently large
disc $D$ containing all the bifurcation values.
\end{proof}
Next, we analyze the Seifert form $S$ associated with the Seifert surface $\Sigma$ defined above.
\begin{proposition}\label{prop:Seinondeg} {\rm (a)} Set  ${\mathcal D}:=
P^{-1}(D)\cap \partial B_R$ and  $\Phi_{I}:=\bigcup_{t\in I}\phi^{-1}(e^{2\pi it})$ for any subset $I\subset [0,1]$. We write
$\Phi_{1}$ for $\Phi_{\{1\}}$ (in particular $\Phi_{1}$ is a Seifert surface).
Then the groups
 $H_n(\Phi_1\cup{\mathcal D},\mathbb{Z})$ and $H_n(\Phi_1,\mathbb{Z})^*$
are isomorphic. In fact one has the following sequence of isomorphisms, denoted by $s$:
$$H_n(\Phi_1\cup{\mathcal D})
\stackrel{\partial^{-1}}{\longrightarrow}
H_{n+1}(S^{2n+1},\Phi_1\cup{\mathcal D})
\stackrel{(1)}{\longrightarrow}
H_{n+1}(S^{2n+1},\Phi_{[0,\frac{1}{2}]}\cup {\mathcal D})
\stackrel{(2)}{\longrightarrow}
$$
$$
H_{n+1}(\Phi_{[\frac{1}{2},1]},\ \partial \Phi_{[\frac{1}{2},1]} )
\stackrel{(3)}{\longrightarrow}
H_n(\intr \Phi_{[\frac{1}{2},1]}  )^*
\stackrel{(4)}{\longrightarrow}
H_n({\Phi_1})^*,  
$$
where $\partial^{-1}$ comes from the exact sequence of the pair,
{\rm (1)} and {\rm (4)} are induced by deformation retracts,
{\rm (2)} is an  excision, while {\rm (3)} is provided by Lefschetz duality.\\
{\rm (b)} Let $j:H_n(\Phi_1,\mathbb{Z})\to H_n(\Phi_1\cup{\mathcal{D}},\mathbb{Z})$ be induced by the inclusion.
Then the composition $$H_n(\Phi_1,\mathbb{Z})\stackrel{j}{\longrightarrow}
H_n(\Phi_1\cup{\mathcal{D}},\mathbb{Z})\stackrel{s}{\longrightarrow}H_n(\Phi_1,\mathbb{Z})^*$$
 can be identified with the Seifert form $S$ associated with $\Phi_1=\Sigma\subset \partial B_R$.\\
{\rm (c)} Let  $b_\infty$ and $h_\infty$ be
the intersection form and monodromy of $H_n(\Phi_1,\mathbb{Z})=
H_n(X_c,\mathbb{Z})$, cf.  \ref{prop:newSeifert}. Then, in matrix notation,
$$b_\infty=-\varepsilon S- S^T \ \  \ \mbox{and } \ \ \ -\varepsilon S^Th_\infty=S.$$
In particular, $h_\infty$ is an automorphism of $S$, that is $h_\infty^TSh_\infty=S$.
\end{proposition}
\begin{proof} Part (a) is clear,
 while part (b) and (c) follow by similar
argument as in the classical case, see e.g. the survey \cite[(3.15)]{Nem-Eger}.
\end{proof}

 \begin{definition}\label{def:stame}
$P$ is called  \emph{spectrally tame} if the following conditions are satisfied:
\begin{itemize}
 \item[(S1)]  $\widetilde{H}_{n-1}(X_c,\mathbb{Q})=0$;
\item[(S2)] the natural inclusion induces an isomorphism
$$H_n(X_c\cap B_R) \to H_n((X_c\cap B_R)\cup (P^{-1}(D)\cap \partial B_R));$$
\item[(S3)] the spectrum of the HVS associated with the link $M_c$ (and the Seifert form
  $S$) is the mod\,2 reduction of the MHS spectrum of $P$ at infinity.
\end{itemize}
\end{definition}
Note that the conditions are `only' (co)homological,
a fact which allows more possibilities for their verifications and for applications. By excision and
the long homological
exact sequence of the pair, $(S2)$ is the consequence of the vanishings (for $c\in\partial D$)
$$\mathrm{(S2')} \hspace{2cm}
H_q(P^{-1}(D)\cap \partial B_R, P^{-1}(c)\cap \partial B_R)=0 \ \ \ \mbox{for $q=n, n+1$.}\hspace{1.5cm}$$
By Proposition \ref{prop:Seinondeg}(b), (S2) is equivalent with the non--degeneracy of the Seifert form $S$.

\begin{remark}
If all the fibers of $P$ have only  isolated singularities, and $P$ is regular at infinity
then  conditions (S1) and $\mathrm{(S2')}$ are satisfied, see e.g.  \cite{Nem-Japan}.
For (S3) one needs additionally the fact that the MHS at infinity is polarized by the
intersection form (for monodromy eigenvalues $\lambda\not=1$) and by the Seifert form
(for eigenvalue  $\lambda=1$). For $\lambda\not=1$ this usually follows from the standard
properties of a projectivization/compactification  of the fibers of $P$ and the Hodge--Riemann
polarization properties, while for $\lambda=1$  one needs an additional Thom--Sebastiani type
argument (that is adding e.g. $z^N$, where $z$ is a new variable) in order to reduce the situation to
the $\lambda\not=1$ case. This means that $P+z^N$ also should
satisfy certain/similar regularity  conditions at infinity.
\end{remark}

\begin{proposition}\label{lem:star}
The  conditions (S1)--(S3)  are guaranteed for
several `tame' polynomials present in the literature:

(a) \ $*$--polynomials \cite{GN0};

(b) \ M--tame polynomials \cite{NZ1};

(c) \ cohomologically tame polynomials \cite{Sa}.
\end{proposition}
\begin{proof} (a) The topological part follows from \cite{GN0}, (S3) from \cite[section 5]{GN}.
(b) The topological part follows from \cite{NZ1,NZ2}, the Hodge theoretical part from \cite{NS}.
(c) follows from \cite{Sa,NS}.
\end{proof}
An immediate application of (S1)--(S2) is the following
\begin{lemma}\label{lem:bettisum}
With the notations of \ref{def:stame},
let $\Sigma\cup X_c$ be the smooth closed manifold obtained from $\Sigma$ and $X_c\cap B_R$ by gluing them
together along their boundary $M_c$. Then
 imply $b_n(  \Sigma\cup X_c)=b_n(\Sigma)+b_n( X_c)$.
\end{lemma}
\begin{proof}
In the long homological exact sequence of the pair $(\Sigma\cup X_c,\Sigma)$ one has
$H_{n+1}(\Sigma\cup X_c,\Sigma)=H_{n+1}(X_c,M_c)=H^{n-1}(X_c)=0$ and
$H_{n-1}(\Sigma)=H_{n-1}(X_c)=0$.
\end{proof}

\section{Semicontinuity results}\label{s:4}
\subsection{Local case}\label{ss:local}
Let $P_t\colon(\mathbb{C}^{n+1},0)\to(\mathbb{C},0)$ be a family of germs of
analytic maps depending smoothly on a parameter $t\in (\mathbb{C},0)$.
Assume that $P_0^{-1}(0)$ has an isolated singularity at $0\in\mathbb{C}^{n+1}$. Let us
 fix a small Milnor ball $B\subset\mathbb{C}^{n+1}$ centered at $0$, and
let $h_0$ be the homological monodromy operator of the Milnor fibration of $P_0$.

\begin{theorem}\label{thm:semicloc}
For any $t$ with  $0<|t|\ll 1$, let $x_1,\dots,x_k$ be all the singular points of $P_t^{-1}(0)\cap B$,
and $\alpha\in [0,1]$ is chosen so that $\xi=e^{2\pi i \alpha}$ is not an eigenvalue of $h_0$. Then
\begin{equation}\label{eq:semicon}
\begin{split}
|\spt_0\cap(\alpha,\alpha+1)|&\geqslant \sum_{j}\ |\spt_j\cap(\alpha,\alpha+1)|,\\
|\spt_0\setminus[\alpha,\alpha+1]|&\geqslant \sum_{j}\ |\spt_j\setminus[\alpha,\alpha+1]|,
\end{split}
\end{equation}
where $\spt_0$ (respectively $\spt_j$) is the mod\,2 reduction of the MHS--spectrum  associated with
the singularity at $0$ of $P_0$ (respectively with the singularity $x_j$ of $P_t$).
\end{theorem}
\begin{proof}
Let $\LS_{x_j}$ be the link of $(P_t^{-1}(0),x_j)$ with Seifert surface (i.e., Milnor fiber)
 $\Sigma_j$, and let $M$ be the link of $(P_0^{-1}(0),0)$ with Seifert surface
$\Sigma $ in $\partial B$.

First assume that $\alpha$ is chosen so that $\xi=e^{2\pi i\alpha}$ is  not
 an eigenvalue of the homological monodromy operators $h_j$ of $(P_t^{-1}(0),x_j)$ for  neither $j=1,\dots,k$.
Then, the nullities $n_{\LS_{x_j}}(\xi)$ and $n_{M}(\xi)$
are all zero (see Remark~\ref{rem:nulnonzero}). Then we apply Proposition~\ref{prop:bol} to obtain
\[\Big|\sigma_{M}(\xi)-\sum_{j}\sigma_{\LS_{x_j}}(\xi)\Big|\leqslant b_n(X^s\cup \Sigma)-b_n(\Sigma)-\sum_{j}b_n( \Sigma_j),\]
where $X^s$ is a smoothing (nearby smooth fiber) of $f^{-1}_0(0)\cap B$.
By removing the absolute value sign, and using $b_n(X^s)=b_n(\Sigma)$ (by \cite{Milnor-book})
and  $b_n(X^s)+b_n(\S)=b_n(X^s\cup\S)$ (by the local analogue of Lemma~\ref{lem:bettisum} proved in the same way)
we obtain
\begin{equation}\label{eq:stepto}
-\sigma_{M}(\xi)+b_n( \Sigma)\geqslant \sum_{j}(-\sigma_{\LS_{x_j}}(\xi)+b_n( \Sigma_j)).
\end{equation}
In this way, via Proposition~\ref{prop:sing}, we prove the first inequality of \eqref{eq:semicon}.
The second one is proved if we
remove the absolute value sign in the other way.

If $\xi=e^{2\pi i\alpha}$ happens to be an eigenvalue of $h_j$ for some $j$,
 but it is not an eigenvalue of $h_0$,
then for all $\alpha'$ sufficiently close to $\alpha$, $e^{2\pi i\alpha'}$ is not an eigenvalue
neither of $h_0$, nor of any $h_j$. We can deduce the inequality \eqref{eq:semicon} for $\alpha$
from the fact that it holds for $\alpha'$ and that the function
$\alpha\to |\spt_j\cap (\alpha,\alpha+1)|$ is lower semicontinuous.
\end{proof}

\subsection{Semicontinuity at infinity}\label{ss:inf}
We prove topologically that the mod\,2 reduction of the result of \cite{NS} is valid for polynomial maps which are `tame' at infinity.
\begin{theorem}\label{thm:semicinf}
Let $P_t$ be a smooth family of spectrally tame polynomial maps, where $t\in ({\mathbb C},0)$.
Then, for any $t$ with  $0<|t|\ll 1$, and for any $\alpha\in [0,1]$ such
that $\xi=e^{2\pi i \alpha}$ is not a root of the homological
monodromy operator at infinity of $P_t$, we have
\begin{align*}
|\spt_t\cap(\alpha,\alpha+1)|\, \geqslant\,&\,|\spt_0\cap (\alpha,\alpha+1)|\\
|\spt_t\setminus[\alpha,\alpha+1]|\, \geqslant\,&\,|\spt_0\setminus [\alpha,\alpha+1]|,
\end{align*}
where $\spt_t$ denotes the mod\,2 spectrum of the MHS at infinity associated with $P_t$.
\end{theorem}
\begin{proof}
We fix a regular fiber $X_0=P_0^{-1}(c)$ of $P_0$.
Let $B$ be a large ball so that $X_0\cap\partial B=M_{\Sigma_0}$ is the link at infinity of $P_0$.
For any $t$ with  $0<|t|\ll 1$, the intersection of $X_t=P_t^{-1}(c)$ with $\partial B$ is isotopic to $M_{\Sigma_0}$. After perturbing $c$, if necessary, we can assume that
$X_t$ is smooth. Let $B_t\supset B$ be a ball such that $X_t\cap\partial B_t$ is the link  at infinity of $X_t$. Then
\[Y=X_t\cap (B_t\setminus \intr B)\]
realizes a cobordism between $M_{\Sigma_0}$ and $M_{\Sigma_t}$.
 The proof is similar to the proof of Theorem~\ref{thm:semicloc} based on Assumptions (S1)--(S4)
 and Proposition~\ref{prop:bol}.
\end{proof}

\subsection{Local to global case}\label{ss:ltg}
We can also compare the mod\,2 spectrum  of all the local
singularities of a fixed fiber with the spectrum at infinity of a spectrally tame polynomial.

\begin{theorem}\label{thm:semiclocinf}
Let $P\colon\mathbb{C}^{n+1}\to\mathbb{C}$ be a spectrally tame polynomial.
Let $X$ be one of its fibers with (isolated) singular points  $x_1,\dots,x_k$.
We denote by $\LS_{x_1},\dots,\LS_{x_k}$ the corresponding links of local singularities and
$\spt_1,\dots,\spt_k$ the mod\,2 reduction of their MHS--spectrum.  Let
$\spt_\infty$ be the mod\,2 reduction of the MHS--spectrum at infinity of $P$.
Then, if $\alpha\in [0,1]$ is chosen so that $\xi=e^{2\pi i\alpha}$
is not an eigenvalue of the monodromy of $P$ at infinity, then
\begin{align*}
 |\spt_\infty\cap(\alpha,\alpha+1)|&\geqslant \sum_{j}\ |\spt_j\cap (\alpha,\alpha+1)|\\
|\spt_\infty\setminus[\alpha,\alpha+1]|&\geqslant \sum_{j}\ |\spt_j\setminus [\alpha,\alpha+1]|.
\end{align*}
\end{theorem}
\begin{proof}
 We follow closely the proof of Theorem~\ref{thm:semicloc} with the modification that now
 $B$ is a large ball such that $X\cap\partial B$ is the link at infinity of $P$.

\end{proof}

\end{document}